\documentclass[11pt,a4paper]{amsart}
\usepackage{enumerate}
\usepackage{amsfonts,amssymb}
\usepackage{amsthm,amsmath,amstext,xfrac}
 \usepackage{tikz}
 \usetikzlibrary{matrix,arrows,calc,decorations.pathreplacing,decorations.markings}
 \usetikzlibrary{decorations.pathmorphing}
 \usetikzlibrary{positioning}

\usepackage{cancel,xspace}
\usepackage{epsfig,fancybox}
\usepackage{graphicx,mathrsfs,ifpdf}
\usepackage{setspace}
\usepackage{cleveref}
\usepackage{enumerate}
\usepackage[colorinlistoftodos,prependcaption,textsize=tiny]{todonotes}

\newtheorem{theorem}{Theorem}[section]
\newtheorem{lemma}[theorem]{Lemma}

\newtheorem{corollary}[theorem]{Corollary}

\theoremstyle{definition}

\numberwithin{equation}{section}

\newcommand{\Pp}{{\mathbb P}^2}

\newcommand{\claimproofend}{\hspace*{.1mm}\hspace{\fill}}

\newenvironment{tikzgraph}
  {\begin{tikzpicture}
      [vertex/.style={circle, draw=black, fill, inner sep=0mm, minimum size=3pt},edge/.style={semithick},
       subdivision/.style={circle, draw=black, fill=white, inner sep=0mm, minimum size=3pt},edge/.style={semithick}]\begin{scope}}
  {\end{scope}\end{tikzpicture}}

\begin{document}
\title{The width of quadrangulations of the projective plane}
\author{Louis Esperet}
\thanks{}
\address{Laboratoire G-SCOP (CNRS, Univ. Grenoble Alpes),
  Grenoble, France}
\email{louis.esperet@grenoble-inp.fr}
\author{Mat\v ej Stehl\'ik} \thanks{Partially supported by ANR project
  Stint under reference ANR-13-BS02-0007 and by the LabEx PERSYVAL-Lab
  (ANR--11-LABX-0025)} \address{Laboratoire G-SCOP, Univ.\ Grenoble Alpes, France}
\email{matej.stehlik@grenoble-inp.fr} \date{\today}

\begin{abstract}
  We show that every $4$-chromatic graph on $n$ vertices, with no two
  vertex-disjoint odd cycles, has an odd cycle of length at most
  $\tfrac12\,(1+\sqrt{8n-7})$. Let $G$ be a non-bipartite quadrangulation
  of the projective plane on $n$ vertices. Our result immediately implies
  that $G$ has edge-width at most $\tfrac12\,(1+\sqrt{8n-7})$, which is
  sharp for infinitely many values of $n$. We also show that $G$ has face-width
  (equivalently, contains an odd cycle transversal of cardinality) at most
  $\tfrac14(1+\sqrt{16 n-15})$, which is a constant away from the optimal; we
  prove a lower bound of $\sqrt{n}$. Finally, we show that $G$ has an odd
  cycle transversal of size at most $\sqrt{2\Delta n}$ inducing a single edge,
  where $\Delta$ is the maximum degree. This last result partially answers a
  question of Nakamoto and Ozeki.
\end{abstract}
\maketitle

\section{Introduction}
\label{sec:introduction}

Erd\H os~\cite{Erd74} asked whether there is a constant $c$ such that every
$n$-vertex $4$-chromatic graph has an odd cycle of length at most $c\sqrt n$.
Kierstead, Szemer\'edi and Trotter~\cite{KST84} proved the conjecture
with $c=8$, and the constant was gradually brought down to $c=2$~\cite{Jia01,Nil99}.
A natural question, asked by Ngoc and Tuza~\cite{NgoTuz95}, is to determine
the infimum of $c$ such that every $4$-chromatic graph on $n$ vertices has an
odd cycle of length at most $c\sqrt n$.

A construction due to Gallai~\cite{Gal63} shows that $c>1$, and this
was subsequently improved to $c>\sqrt 2$ by Ngoc and Tuza~\cite{NgoTuz95},
and independently by Youngs~\cite{You96}. The graphs they used---the so-called
\emph{generalized Mycielski graphs}---are a subclass of a rich family of graphs
known as \emph{non-bipartite projective quadrangulations}. These are graphs that
embed in the projective plane so that all faces are bounded by four edges, but
are not bipartite. This family of graphs plays an important role in the
study of the chromatic number of graphs on surfaces: it was shown by Youngs~\cite{You96} that all such
graphs are $4$-chromatic. Gimbel and Thomassen~\cite{GimTho97} later proved
that triangle-free projective-planar graphs are $3$-colorable if and
only if they do not contain a non-bipartite projective quadrangulation,
and used this to show that the $3$-colorability of triangle-free
projective-planar graphs can be decided in polynomial time.
Thomassen~\cite{Tho04} also used projective quadrangulations to give
negative answers to two questions of Bollob\'as~\cite{Bol78} about
$4$-chromatic graphs.

Let $G$ be a non-bipartite projective quadrangulation. A key
property of such a graph is that any cycle in $G$ is contractible on the
surface if and only if it has even length. In particular, the length of a
shortest odd cycle in $G$ is precisely the \emph{edge-width} of $G$,
the length of a shortest non-contractible cycle. Since any two
non-contractible closed curves on the projective plane intersect, it
also follows that $G$ does not contain two vertex-disjoint odd cycles.
The interest in the study of odd cycles in $4$-colorable graphs also
comes from the following question of Erd\H os~\cite{Erd68}: does every
$5$-chromatic $K_5$-free graph contain a pair of vertex-disjoint odd
cycles? Erd\H os's question may be rephrased as follows: is every
$K_5$-free graph without two vertex-disjoint odd cycles
$4$-colorable? This was answered in the affirmative by Brown and
Jung~\cite{BroJun69}. The non-bipartite projective quadrangulations
provide an infinite family of graphs showing that $4$ cannot be
replaced by $3$.  Note that Erd\H os's question was generalized by
Lov\'asz and became known as the Erd\H os--Lov\'asz Tihany
Conjecture. So far only a few cases of the conjecture have been
proved.

\smallskip

Our first theorem settles the problem of Ngoc and Tuza for the case of
$4$-chromatic graphs with no two vertex-disjoint odd cycles (and in
particular, for non-bipartite projective
quadrangulations).

\begin{theorem}\label{thm:oddcycle2}
Let $G$ be a $4$-chromatic graph on $n$ vertices without two
vertex-disjoint odd cycles. Then $G$ contains an odd cycle of length
at most $\tfrac12(1+\sqrt{8n-7})$.
\end{theorem}

Note that in the generalized Mycielski graphs found by Ngoc and Tuza~\cite{NgoTuz95},
and independently by Youngs~\cite{You96}, the shortest odd cycles have precisely this number
of vertices, so Theorem~\ref{thm:oddcycle2} is sharp for infinitely
many values of $n$.

\smallskip

An \emph{odd cycle transversal} in a graph $G$ is a set of vertices
$S$ such that $G-S$ is bipartite. Since any two odd cycles intersect
in a non-bipartite projective quadrangulation $G$, it follows that
any odd cycle in $G$ is also an odd cycle transversal of $G$. The following
slightly more general result holds. If $\gamma$ is a non-contractible
closed curve whose intersection with $G$ is a subset $S \subseteq V(G)$
(the minimum size of such a set $S$ is called the \emph{face-width} of $G$),
then $G-S$ is bipartite. It follows that the minimum size of an odd cycle
transversal of $G$ cannot exceed the face-width of $G$, and it can be proved
that the two parameters are indeed equal.

\Cref{thm:oddcycle2} immediately implies that a non-bipartite
projective quadrangulation on $n$ vertices has an odd cycle
transversal with at most $\tfrac12(1+\sqrt{8n-7}) \approx \sqrt{2n}$ vertices. Our next
theorem improves the bound to roughly $\sqrt{n}$.

\begin{theorem}
\label{thm:OCT-upper}
Let $G$ be a non-bipartite projective quadrangulation on $n$
vertices. Then $G$ has an odd cycle transversal of size at most
$\left\lfloor\tfrac14+\sqrt{n-\tfrac{15}{16}}\right\rfloor$.
\end{theorem}

The next result shows that this is close to optimal.

\begin{theorem}
\label{thm:OCT-lower}
For any integer $k\ge 2$, there is a non-bipartite
  projective quadrangulations on $n=k^2$ vertices containing no odd cycle
  transversal of size less than $k=\sqrt n$.
\end{theorem}

It can be checked that for any integer $k\ge 2$,
$\left\lfloor\tfrac14+\sqrt{k^2-\tfrac{15}{16}}\right\rfloor=k$, so this shows that
Theorem~\ref{thm:OCT-upper} is sharp for infinitely many values of $n$.

\smallskip

Nakamoto and Ozeki~\cite{NO17} have asked whether every
$n$-vertex non-bipartite projective quadrangulation can be
$4$-colored so that one color class has size $1$ and another has
size $o(n)$. While we were unable to answer their question in general,
our final theorem gives a positive answer when the maximum degree is
$o(n)$.

\begin{theorem}\label{thm:sqrtD}
Let $G$ be a non-bipartite projective quadrangulation on $n$
vertices, with maximum degree $\Delta$. There exists an odd cycle
transversal of size less than $\sqrt{2\Delta n}$ inducing a single edge.
\end{theorem}

The rest of the paper is organized as follows. In
\Cref{sec:preliminaries} we introduce the necessary terminology and
prove a number of lemmas that will be used later.  In
\Cref{sec:oddcycle} we prove \Cref{thm:oddcycle2} using a theorem of
Lins on graphs embedded in the projective plane (an equivalent form of
the Okamura-Seymour theorem~\cite{OS81} in Combinatorial Optimisation), and the
Two Disjoint Odd Cycles Theorem of Lov\'asz (see~\cite{KawOze13,Sey95}).  In \Cref{sec:OCT} we
prove \Cref{thm:OCT-upper} using a theorem of Randby~\cite{Ran97}, and
then show that a similar result holds for any $4$-vertex-critical graph in
which any two odd cycles intersect. As a consequence of
\Cref{thm:OCT-upper}, we also deduce an (almost tight) lower bound on the independence
number of non-bipartite projective quadrangulations.  In
Section~\ref{sec:jap}, we prove Theorem~\ref{thm:sqrtD} and finally,
we conclude with some open problems in Section~\ref{sec:conclusion}.

\section{preliminaries}
\label{sec:preliminaries}

Our graph theoretic terminology is standard, and follows Bondy and
Murty~\cite{BonMur08}. For the notions from algebraic topology, the
reader is referred to~\cite{Mun84}.

We denote the (real) projective plane by $\Pp$.
We will use the following properties of the projective plane, which
can be proved using basic algebraic topology. Namely, every simple closed
curve $\gamma:[0,1] \to \Pp$ is either nullhomotopic or non-contractible,
and two non-contractible essential simple closed curves intersect transversally
an odd number of times.

Given a non-bipartite projective quadrangulation $G$, it is not hard to show
(see e.g. ~\cite[Lemma 3.1]{KaiSte15}) that all contractible closed walks
in $G$ have even length, and all non-contractible closed walks in $G$ have
odd length.

Given a quadrangulation $G=(V,E)$ of the projective plane $\Pp$, a
\emph{support set} $S$ is a circularly ordered subset of $V$ (with
repetitions allowed), such that any two consecutive vertices of $S$
are on a common face of $G$. Since $G$ is a quadrangulation, it
follows that two distinct consecutive vertices of $S$ are either adjacent or
have a common neighbor (in this case, we say that they are
\emph{opposite}). The \emph{size} of $S$ is the number of pairs of
consecutive vertices of $S$, the \emph{order} of $S$ is the number of
pairs of consecutive vertices of $S$ that are adjacent in $G$, and the
\emph{parity} of $S$ is the parity of the order of $S$. Note that
to any support set $S$ of $G$ we can associate a closed curve of $\Pp$
meeting $G$ in $S$, and encountering the vertices of $S$ in their
circular order. We denote such a curve by $\rho(S)$. Conversely, to
any curve $\rho$ meeting $G$ in a subset $S\subseteq V$ we can
associate a support set $S(\rho)$ whose circular order coincides with
the order in which $\rho$ visits the vertices of $S$.

Given a support set $S$ of $G$, and two consecutive vertices
$v_i,v_{i+1}$ of $S$ that are opposite, a \emph{shift} in $S$ at
$(v_i,v_{i+1})$ is the support set obtained from $S$ by adding a
common neighbor $u_i$ of $v_i$ and $v_{i+1}$ between $v_i$ and
$v_{i+1}$ in the support set ($u_i$ may be chosen arbitrarily among
the common neighbors of $v_i$ and $v_{i+1}$ in $G$). Observe that
replacing a pair of opposite vertices by two pairs of adjacent
vertices does not change the parity of the support set. Moreover, if
$S'$ is obtained from $S$ by a shift, then $\rho(S)$ and $\rho(S')$
are homologous. For convenience, we write it as a lemma.

\begin{lemma}\label{lem:shift}
Let $G$ be a non-bipartite quadrangulation of $\Pp$ and $S$ a
support set of $G$. If a support set $S'$ is obtained from $S$ by a
sequence of shifts, then $S$ and $S'$ have the same parity and
$\rho(S)$ and $\rho(S')$ are homologous.
\end{lemma}

Recall that a closed walk
in a non-bipartite quadrangulation of $\Pp$ is odd if and only if it
is non-contractible. The next lemma shows a similar result for
support sets.

\begin{lemma}\label{lem:parity}
If $G$ is a non-bipartite quadrangulation of $\Pp$ and $S$ a
support set of $G$, then $S$ is odd if and only if $\rho(S)$ is
non-contractible. In particular, if $S$ is odd, then $G-S$ is
bipartite.
\end{lemma}

\begin{proof}
 For any two consecutive vertices $v_i$ and $v_{i+1}$ of $S$ that are
 opposite, we make a shift at $(v_i,v_{i+1})$. Let $S'$ be the support
 set thus obtained. By Lemma~\ref{lem:shift}, $S$ and $S'$ have the
 same parity and $\rho(S)$ and $\rho(S')$ are homologous. By the
 definition of $S'$, any two consecutive vertices are adjacent. It
 follows that $S'$ is a closed walk in $G$. Since a closed walk in a
 non-bipartite quadrangulation of $\Pp$ is odd if and only if it is
 non-contractible, $S$ is odd if and only if $\rho(S)$ is
 non-contractible.

 Assume now that $S$ is odd, and so $\rho(S)$ is non-contractible. Then
 the removal of $S$ yields a graph embedded in the plane, with all inner
 faces bounded by an even number of edges. Therefore, $G-S$ is bipartite.
\end{proof}

\begin{lemma}\label{lem:shortest}
Let $G$ be a non-bipartite quadrangulation of $\Pp$ and $S$ an odd
support set of $G$. Then there is an odd support set $S' \subseteq S$,
with order at most the order of $S$, such that the vertices of $S'$
are pairwise distinct, and two vertices of $S'$ are adjacent or
opposite if and only if they are consecutive.
\end{lemma}

\begin{proof}
Let $S'\subseteq S$ be an odd support set of order at most the order
of $S$, and (with respect to these properties) with minimum
size. Assume first that some vertex appears at least twice in
$S'$. Then we can divide $S'$ into two support sets of different
parities. In particular there is an odd support set $S''\subseteq S'$,
of order at most the order of $S'$, and of size less than the size of
$S'$. This contradicts the minimality of $S'$.

Assume now that two non-consecutive vertices $u$ and $v$ of $S'$ are
opposite or adjacent. Again, we can divide $S'$ into two support sets
(where $u$ and $v$ are now consecutive), of order at most the order of
$S'$ plus one. Since $S'$ is odd, the two support sets have
different parities, so one of them is odd (and therefore has order at
most the order of $S'$), which contradicts the minimality of
$S'$.
\end{proof}

The same proof also gives the following similar result, which will be needed
later.

\begin{lemma}\label{lem:shortest2}
Let $G$ be a non-bipartite quadrangulation of $\Pp$ and $S$ an odd
support set of $G$ corresponding to a non-contractible closed walk in $G$. Then $G$
contains an odd cycle with vertex set $S' \subseteq S$.
\end{lemma}

\section{Short odd cycles}
\label{sec:oddcycle}

A key result is the following corollary of a theorem of Lins~\cite{Lin81}
(see also~\cite[Corollary 2.4]{ArcBon97}).

\begin{lemma}\label{lem:lins}
  Let $G=(V,E)$ be a projective planar graph with a shortest non-contractible cycle of
  length $\ell$ and a shortest non-contractible cycle of length $\ell^*$ in its dual
  graph $G^*$. Then $|E|\ge \ell \cdot \ell^*$.
\end{lemma}

We now show that the following result follows from
Lemma~\ref{lem:lins} as a fairly simple consequence.

\begin{theorem}\label{thm:oddcycle}
Let $G$ be a non-bipartite projective quadrangulation on $n$
vertices. Then $G$ contains an odd cycle of length at most $\tfrac12(1+\sqrt{8n-7})$.
\end{theorem}

\begin{proof}
  A standard application of Euler's formula shows that $G$ has
  $m=2n-2$ edges. Let
  $\ell$ be the length of a shortest odd cycle in $G$. By
  Lemma~\ref{lem:lins}, we know that the dual graph $G^*$ of $G$
  contains a non-contractible cycle of length $\ell^* \le
  (2n-2)/\ell$. Let $C^*$ be such a cycle, and let
  $(f_1,f_2,\ldots,f_{\ell^*})$ be the faces of $G$ corresponding to the
  vertices of $C^*$. Note that any two consecutive faces $f_i,f_{i+1}$
  in $C^*$ share an edge, which we call $e_i$. For any $1\le i \le
  {\ell^*}+1$, we will choose a vertex $v_i$
  in each edge $e_i$ (where we set $e_{{\ell^*}+1}=e_1$), in a specific way. We start by choosing $v_1$ in
  $e_1$ arbitrarily, and for any $i>1$ we distinguish two cases. If
  $v_{i-1} \in e_i$, then we set $v_i=v_{i-1}$, and otherwise we
  choose for $v_i$ a vertex adjacent to $v_{i-1}$ in $f_i$ (note that
  such a vertex always exists). Let $S$ be the sequence of vertices thus
  chosen (each maximal sequence of consecutive vertices
  $v_i,v_{i+1},\ldots,v_{j}$ such that $v_i=v_{i+1}=\cdots =v_{j}$ is
  reduced to a single vertex $v_i$).
  Any two consecutive vertices are
  adjacent (for $v_{\ell^*+1},v_1$, this follows from the fact
  that they both belong to $e_{{\ell^*}+1}=e_1$), so we obtain a closed walk in $G$ of length at most
  $\ell^*+1$ that is homotopic to $C^*$, and therefore non-contractible.
  It follows that $S$ is an odd support set where any
  two consecutive vertices are adjacent. By Lemma~\ref{lem:shortest2},
  $G$ contains an odd cycle of length at most $\ell^*+1\le
  1+(2n-2)/\ell$. It follows that $\ell^2-\ell \le 2n-2$, so
  $\ell\le \tfrac12(1+\sqrt{8n-7})$, as desired.
\end{proof}

Following~\cite{KawOze13,Yu03}, a \emph{$k$-separation} in a graph $G$ is a pair $(G_1,G_2)$ of
subgraphs of $G$ such that $V(G) = V(G_1) \cup V(G_2)$, $|V(G_1)\cap
V(G_2)|=k$, $E(G_1)\cap E(G_2)=\emptyset$, and $E(G_i)\cup (V(G_i) -
V(G_{3-i}))\neq \emptyset$ for $i=1,2$. A graph $G$ is said to be
\emph{internally $4$-connected} if $G$ is 3-connected and for every
3-separation $(G_1, G_2)$ in $G$, $|V(G_1)|\le 4$ or $|V(G_2)|\le 4$.

The following characterisation of graphs without two vertex-disjoint
odd cycles will play a key role in our proofs. It was first proved by
Lov\'asz using Seymours's characterisation of regular matroids (see~\cite{Sey95}).
A simpler proof was recently given by Kawarabayashi and Ozeki~\cite{KawOze13}.

\begin{theorem}
\label{thm:disjoint-cycle}
  Let $G$ be an internally $4$-connected graph. Then $G$ has no two
  vertex-disjoint odd cycles if and only if $G$ satisfies one of the
  following conditions:
  \begin{enumerate}
    \item $G-v$ is bipartite, for some $v \in V(G)$;
    \item $G-\{e_1,e_2,e_3\}$ is bipartite for some edges $e_1, e_2, e_3 \in E(G)$
      such that $e_1, e_2, e_3$ form a triangle;
    \item $|V(G)| \leq 5$;
    \item $G$ can be embedded into the projective plane so that every face boundary
      has even length.
  \end{enumerate}
\end{theorem}

In order to extend Theorem~\ref{thm:oddcycle} to
$4$-chromatic graphs without two vertex-disjoint odd cycles, we will
need the following technical lemma about precoloring
extension in bipartite graphs.

\begin{lemma}\label{lem:ext}
Let $G$ be a bipartite graph and $X$ be a subset of $V(G)$ of size
at most $3$. Then any (proper) precoloring of $X$
extends to a $3$-coloring of $G$, unless
\begin{enumerate}[{\normalfont (i)}]
\item $X =\{x,y,z\}$ for distinct $x,y,z$,
\item $x$, $y$, and $z$ are on the same side of the bipartition of $G$,
\item in the precoloring, $x$, $y$ and $z$ have pairwise different colors, and
\item any pair of vertices in $\{x,y,z\}$ have a common neighbor in $G$.
\end{enumerate}
\end{lemma}

\begin{proof}
Let $A,B$ be the bipartition of $G$. By symmetry, we can assume that
$|A\cap X | \ge |B\cap X|$. In particular, $|B\cap X|\le 1$. If in the precoloring of $X$, $A$
contains at most two different colors, then we color each vertex of $A-X$ with one of these two
colors (one that does not appear in $B\cap X$), and each vertex of $B-X$ with the third color. This yields a
(proper) 3-coloring of $G$. Otherwise, by
symmetry, $X=\{x,y,z\}\subseteq A$ and $x,y,z$ have distinct colors; this
proves conditions (i)--(iii) of the lemma.

Assume that $x$ and $y$ have no common neighbor
in $G$ (and thus in $B$), and $x,y,z$ are colored $1,2,3$
respectively. Then we color each vertex of $A-X$ with color 3, each
neighbor of $x$ with color 2, and the remaining vertices of $B$ with
color 1. Since $x$ and $y$ have no common neighbor, this is a proper
3-coloring of $G$ extending the precoloring of $X$.
\end{proof}

A $k$-chromatic graph is said to be \emph{$k$-vertex-critical} if for any vertex $v$,
$G-v$ is $(k-1)$-colorable. A graph is \emph{projective} if it can be embedded in $\Pp$.
We will need the following direct consequence of a result of Gimbel and
Thomassen~\cite[Theorem 5.4]{GimTho97}:

\begin{lemma}\label{lem:GT}
Let $G$ be a simple triangle-free projective graph. If $G$ is $4$-vertex-critical,
then $G$ is a (non-bipartite) projective quadrangulation.
\end{lemma}

\begin{proof}
  Since $G$ is a $4$-chromatic triangle-free projective graph, $G$ contains
  a non-bipartite projective quadrangulation $H$ as a subgraph by a result
  of Gimbel and Thomassen~\cite[Theorem 5.4]{GimTho97}. Since $H$ is itself
  $4$-chromatic and $G$ is vertex-critical, $H$ must be a spanning
  subgraph of $G$. If $H$ is a proper subgraph of $G$, then there is an edge
  $e \in E(G) \setminus E(H)$. Both end vertices of $e$ must lie on the
  boundary of the same face in some embedding of $H$ in $\Pp$, so
  adding $e$ to $H$ creates a triangle or a pair of parallel edges, contradicting
  the hypothesis of the lemma. Therefore $G=H$, so $G$ is a non-bipartite
  projective quadrangulation.
\end{proof}

We now extend Theorem~\ref{thm:oddcycle} to
$4$-chromatic graphs without two vertex-disjoint odd cycles.

\begin{proof}[Proof of Theorem~\ref{thm:oddcycle2}]
  We prove the result by induction on $n$.  We can assume that $G$ is
  $4$-vertex-critical. In
  particular, $G$ is connected and does not contain a clique cutset (a
  clique whose removal disconnects the graph). We may assume that $G$
  has at least six vertices (otherwise $G$ has four or five vertices
  and contains a triangle and the result clearly holds). As a
  consequence, we may also assume that $G$ is triangle-free, since
  otherwise $G$ has an odd cycle of length
  $3\le \tfrac12(1+\sqrt{8n-7})$.

\smallskip

Assume first that $G$ is internally $4$-connected. In this case we can
apply Theorem~\ref{thm:disjoint-cycle}. As $G$ is $4$-chromatic, triangle-free, and
has at least six vertices, none of cases (i)--(iii) applies. It follows
that $G$ can be embedded into the projective plane. Since $G$ is
  $4$-vertex-critical and triangle-free, by Lemma~\ref{lem:GT} it
  is a non-bipartite projective quadrangulation and the result follows
  directly from Theorem~\ref{thm:oddcycle}.

  \smallskip

  Assume now that $G$ is not internally $4$-connected. Since $G$ has no
  clique-cutset, it is $2$-connected. Hence there exist graphs
  $G_1=(V_1,E_1)$ and $G_2=(V_2,E_2)$, and sets $X_i \in V_i$ of two or
  three vertices ($i=1,2$) with $|X_1|=|X_2|$, such that $G_1[X_1]$ and
  $G_2[X_2]$ are equal (as labelled graphs), and $G$ can be obtained
  from $G_1,G_2$ by idenfying $X_1$ in $G_1$ and $X_2$ in $G_2$ (call
  $X$ the corresponding set of vertices in $G$, inducing the same
  graph as $G_1[X_1]$ and $G_2[X_2]$). Moreover, if
  $|X_1|=|X_2|=3$, then for $i=1,2$, $V_i - X_i$ contains at least two
  vertices (this follows from the definition of internal
  $4$-connectivity). Note that since $G$ is triangle-free, $X$ is
  bipartite. In what follows, by a slight abuse of notation, we give
  the same name to a vertex of $X_1$, the vertex of $X_2$ it is
  indentified with, and the resulting vertex of $X$.

 Assume that $|X|=2$, say $X=\{x,y\}$. Since
$G$ has no clique-cutset, $x$ and $y$ are non-adjacent. If $G$
contains an odd cycle disjoint from $X$ (say in $G_2- X_2$), then since $G$ has no two
vertex-disjoint odd cycles, $G_1$ is bipartite. Since $G$ is
$4$-vertex-critical, $G_2$ is 3-colorable and by Lemma~\ref{lem:ext} any
$3$-coloring of $G_2$ extends to $G_1$, a contradiction. It follows that every odd cycle of $G$ intersects
$X$, so $G-X$ is bipartite. As $X$ is a stable set, $G$ is
$3$-colorable, which is a contradiction.

We can now assume that $|X|=3$, say $X=\{x,y,z\}$. Assume first that
$X$ is a stable set. If all odd cycles of $G$ intersect $X$, then
$G-X$ is bipartite and since $X$ is stable, $G$ is $3$-colorable, a
contradiction. Otherwise $G$ contains an odd cycle disjoint from $X$
(say in $G_2- X_2$). Then $G_1$ is bipartite, and by
Lemma~\ref{lem:ext}, $x,y,z$ are on the same side of the bipartition of
$G_1$, and in each $3$-coloring of $G_2$ they have three distinct
colors. Moreover, each pair of vertices among $x,y,z$ has a common
neighbor in $G_1$. Let $H$ be the graph obtained from $G_2$ by adding
a vertex $v$ adjacent to $x,y,z$. Since $G_1-X_1$ contains at least
two vertices, $H$ has less vertices than $G$. Note that $H$ is
$4$-chromatic (since for any $3$-coloring of $H-v$, the neighbors of $v$
have three distinct colors), and has no two vertex-disjoint odd cycles:
each odd cycle $C$ of $H$ is either disjoint from $v$, and is
therefore an odd cycle of $G$, or intersects $X$ in two vertices, say
$x$ and $y$, and corresponds to an odd cycle of $G$ coinciding with
$C$ in $G_2$ and whose intersection with $G_1-X$ is a single common
neighbor of $x$ and $y$ in $G_1$ (which is known to exist). By the
induction hypothesis, $H$ (and therefore $G$) has an odd cycle of
length at most $\tfrac12(1+\sqrt{8n-7})$.

Now assume that $X$ is not a stable set.
Since $G$ is triangle-free, we can assume that $G[X]$ has two
non-adjacent vertices, say $y,z$, while $x$ is adjacent to at least one
of them, say $y$. By Lemma~\ref{lem:ext}, none of $G_1,G_2$ is
bipartite and in particular, every odd cycle intersects $X$. Note that
$G-\{y,z\}$ is not bipartite (since otherwise $G$ would be
$3$-colorable), so we can assume that $G_2$ has an odd cycle $C_2$
containing $x$ and avoiding $y,z$. Therefore every odd cycle of
$G_1$ intersects $x$, so $G_1-x$ is bipartite, say with bipartition
$A,B$. Take a $3$-coloring $c$ of $G_2$, and assume without loss of
generality that $x$ and $y$ have colors
$1$ and $2$, respectively. If $y,z$ are both in $A$ and $x,y,z$ do not have pairwise distinct
colors, then $c$ easily extends to a $3$-coloring of $G_1$. Similarly,
if $y,z$ have distinct colors and are in distinct partite sets, then
$c$ easily extends to a $3$-coloring of $G_1$. So we can assume that
either
\begin{enumerate}
\item $y,z \in A$, $y$ has color $2$, and $z$ has color $3$, or
\item $y\in A$, $z\in B$, and $y,z$ are both colored $2$.
\end{enumerate}
Let $B_x$ be the set of
neighbors of $x$ in $B$.  We color $A-y$ with color $3$, $B_x$ with
color $2$, and $B-(B_x\cup X)$ with color $1$. Since $G$ is triangle-free
there are no edges between $y$ and $B_x$, so the resulting $3$-coloring
of $G$ is proper, which contradicts the fact that $G$ is
$4$-chromatic. This concludes the proof of Theorem~\ref{thm:oddcycle2}.
\end{proof}

\section{Small odd cycle transversals}
\label{sec:OCT}

A  (multi)graph $G$ embedded in a surface $\Sigma$ is \emph{minimal of
  face-width $k$} if the face-width of $G$ is $k$, while for any edge
$e$ of $G$, the face-width of $G/e$ (the (multi)graph embedded in $\Sigma$
obtained from $G$ be contracting $e$) and the face-width of $G-e$ are
less than $k$.

We will use the following result of Randby~\cite{Ran97}:

\begin{theorem}\label{thm:ran97}
For any integer $k$, if a multigraph embedded in $\Pp$ is minimal of
  face-width $k$, then it contains exactly $2k^2-k$ edges.
\end{theorem}

\begin{proof}[Proof of Theorem~\ref{thm:OCT-upper}]
Let $G$ be a non-bipartite quadrangulation on $n$ vertices of the
projective plane $\Pp$, and let $k$ be the face-width of $G$.  By
Euler's formula, $G$ has $m=2n-2$ edges. If $G$ is not minimal with
face-width $k$, we delete or contract edges of $G$ until we obtain a
(multi)graph $H$ that is minimal with face-width $k$. Note that $H$
has at most $m=2n-2$ edges by construction, and exactly $2k^2-k$ edges
by Theorem~\ref{thm:ran97}. It follows that $2k^2-k\le 2n-2$ and so
$k\le \tfrac14+\sqrt{n-\tfrac{15}{16}}$, as desired.
\end{proof}

We believe that Theorem~\ref{thm:OCT-upper} can be extended to
$4$-chromatic graphs with no two vertex-disjoint odd cycles, in the same
way Theorem~\ref{thm:oddcycle2} extends
Theorem~\ref{thm:oddcycle}. However, we have only been able to prove
that $4$-vertex-critical graphs with no two vertex-disjoint odd cycles satisfy
the result.

\begin{theorem}
\label{thm:OCT-upper2}
  Let $G$ be a $4$-vertex-critical graph on $n$ vertices without two
  vertex-disjoint odd cycles. Then $G$ has an odd cycle transversal of
  cardinality at most $\tfrac14+\sqrt{n-\tfrac{15}{16}}$.
\end{theorem}

\begin{proof}
  The proof proceeds by induction on $n$.  We can assume that $G$ has
  at least nine vertices (by checking small $4$-vertex-critical graphs, for
  instance using~\cite{CGSZ15}) and thus has no odd cycle transversal on
  at most three vertices. Since every odd cycle is an odd cycle
  transversal in $G$, in particular this
  implies that we can assume that $G$ is triangle-free. If $G$ is
  not internally $4$-connected then there exist graphs $G_1=(V_1,E_1)$
  and $G_2=(V_2,E_2)$, and sets $X_i \subseteq V_i$ with at most three vertices
  ($i=1,2$) with $|X_1|=|X_2|$, such that $G_1[X_1]$ and $G_2[X_2]$
  are equal (as labelled graphs), and $G$ can be obtained from
  $G_1,G_2$ by identifying $X_1$ in $G_1$ and $X_2$ in $G_2$ (let $X$
  be the corresponding set of vertices in $G$, inducing the same graph as
  $G_1[X_1]$ and $G_2[X_2]$). Moreover, $G_1,G_2$ have the property
  that if $|X_1|=|X_2|=3$, then for $i=1,2$, $V_i - X_i$ contains at
  least two vertices.

If every odd cycle of $G$ intersects $X$, then $X$ is an odd cycle
transversal of size at most 3, which is a contradiction. It follows
that $G$ contains an odd cycle disjoint from $X$ (say in
$G_2-X_2$). Since any two odd cycles intersect, $G_1$ is bipartite.
Recall that $G$ is $4$-vertex-critical,
so $G_2$ is 3-colorable, and since no 3-coloring of $G_2$ extends to
$G_1$ (otherwise $G$ would be 3-colorable), by
Lemma~\ref{lem:ext}, we have (i) $X =\{x_1,x_2,x_3\}$ for distinct
$x_1,x_2,x_3$, (ii) $x_1$, $x_2$, and $x_3$ are on the same side of
the bipartition of $G_1$, (iii) in any 3-coloring of $G_2$, $x_1$,
$x_2$ and $x_3$ have pairwise different colors, and (iv) any pair of
vertices in $\{x_1,x_2,x_3\}$ have a common neighbor in $G_1$. If
there is a vertex $y$ in $G_1$, adjacent to each of $x_1,x_2,x_3$,
then the graph obtained from $G$ by removing any vertex of $V_1-X_1$
distinct from $y$ is $4$-chromatic, which contradicts the fact that
$G$ is $4$-vertex-critical (since $G_1-X_1$ contains at least
two vertices). It follows that no vertex of $G_1$ is adjacent to each
of $x_1,x_2,x_3$. So there is a set $Y=\{y_1,y_2,y_3\}$ of three vertices in $G_1$, such
that $y_1$ is adjacent to $x_2$ and $x_3$, $y_2$ is adjacent to $x_1$
and $x_3$, and $y_3$ is adjacent to $x_1$ and $x_2$. Since a
$4$-vertex-critical graph has minimum degree at least 3, $G_1$ contains at
least seven vertices. If $G_1$ has at least eight vertices, then remove
from $G$ all the vertices of $V_1-(X\cup Y)$, and add a vertex $z$
adjacent to $y_1,y_2,y_3$. The resulting graph $H$
is $4$-vertex-critical, has no two vertex-disjoint odd cycles, and is smaller
than $G$. By the induction hypothesis, $H$ has an odd cycle transversal $T$ with at
most $\tfrac14+\sqrt{n-\tfrac{15}{16}}$ vertices. Note that
$z$ does not appear in a minimum odd cycle transversal of $H$, so we
can assume that $z\not\in T$. It is easy to check that $T$ is
also an odd cycle transversal of $G$.

\smallskip

\begin{figure}[htbp]
\centering \includegraphics[scale=1]{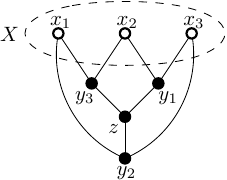}
\caption{A 7-vertex bipartite graph.} \label{fig:7vg}
\end{figure}

By the above paragraph, we can assume that for any decomposition of
$G$ into $G_1$ and $G_2$ as above, on some set $X$ of at most three
vertices, $G_1$ induces the graph on seven vertices in Figure~\ref{fig:7vg}. In
particular, $X$ induces a stable set of size $3$. Moreover, it is not hard to
check that no vertex cutset of $G$ of size at most $3$ intersects $G_1-X$,
otherwise $G$ would contain a vertex cutset of size at most $2$. It
follows that $G$ can be constructed from some graph $G_0$ and a family
$t_1,t_2,\ldots,t_k$ of triples of vertices of $G_0$ by pasting the $X$-part
of a copy $G_i$ of the graph of Figure~\ref{fig:7vg} onto each triple $t_i$;
see Figure~\ref{fig:expl}, top left. Let $H$ be the graph obtained from
$G_0$ by adding, for each $1\le i \le k$, a vertex $z_i$ adjacent to
the vertices of $t_i$; see Figure~\ref{fig:expl},
bottom left. Observe that $H$ is $4$-vertex-critical (otherwise $G$
would not be $4$-vertex-critical), triangle-free (since otherwise $G$
would also contain a triangle) and any two odd cycles
intersect. We claim that $H$ is also internally $4$-connected. To see
this, consider a $3$-separation in $H$ on some vertex cutset $X$, such
that the two parts of the separation contain at least 5 vertices. If $X$
does not intersect any vertex $z_i$, then it is also a vertex cutset
in $G$, and one of the two sides of the corresponding separation in
$G$ is isomorphic the graph of Figure~\ref{fig:7vg}. By the definition
of $H$, this implies that
in $H$, the corresponding side of the separation contains 4 vertices,
which is a contradiction. So we can assume that $X$ contains some
vertex $z_i$. Note that one of the three neighbors of $z_i$ in $H$
(call it $u_i$) is
alone in its part of the separation corresponding to $X$, so by
replacing $z_i$ by $u_i$ in $X$, we obtain a new vertex cutset of
size 3 in $H$. By replacing in this way all the vertices $z_i$ of $X$
by original vertices of $G$, we obtain a new vertex cutset $X'$ of
size 3 in $H$, which is also a vertex cutset in $G$. By definition,
one of the two parts of the corresponding separation in $G$ is the
graph of Figure~\ref{fig:7vg}, which in $H$ corresponds to a vertex
with three pairwise non-adjacent neighbors. By construction, one of
these neighbors has degree two in $H$, which contradicts the fact that
$H$ is 4-vertex-critical, and concludes the proof that $H$ is
internally $4$-connected.

By Theorem~\ref{thm:disjoint-cycle}, $H$ has an embedding
in the projective plane, and by Lemma~\ref{lem:GT}, $H$ is a non-bipartite
quadrangulation of the projective plane. We now replace each $z_i$ by
$G_i-X$ (see Figure~\ref{fig:expl}, right) and observe that $G$ is itself a quadrangulation
of the projective plane. By Theorem~\ref{thm:OCT-upper}, $G$ has an
odd cycle transversal with at most $\tfrac14+\sqrt{n-\tfrac{15}{16}}$ vertices, which concludes the proof.
\end{proof}

\begin{figure}[htbp]
\centering \includegraphics[scale=1]{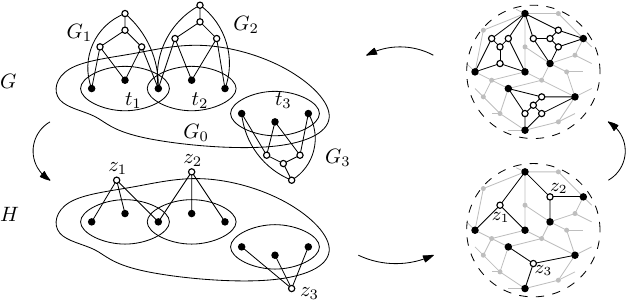}
\caption{Graphs $G$, $H$, and their representations as projective quadrangulations.} \label{fig:expl}
\end{figure}


We now prove that the bound in Theorem~\ref{thm:OCT-upper} is at most $\tfrac14$
away from the optimum for any value of $n$, and is sharp for
infinitely many values of $n$.

\medskip

For a positive integer $k$, let $[k]$ denote the set
$\{0,\ldots,k-1\}$. Let $P_k$ be a path with vertex set $[k]$, with
vertices in the increasing order along $P_k$.
For $k\geq 2$, we define $G_k$ as the graph obtained from the
Cartesian product $P_k \Box P_k$ by adding the edges joining $(0,j)$
to $(k-1,k-j-1)$, and those joining $(j,0)$ to $(k-j-1,k-1)$. The
graphs $G_k$ embed as quadrangulations in $\Pp$; see \Cref{fig:grids}
for embeddings of $G_2$, $G_3$ and $G_4$ in $\Pp$.

\begin{figure}[ht]
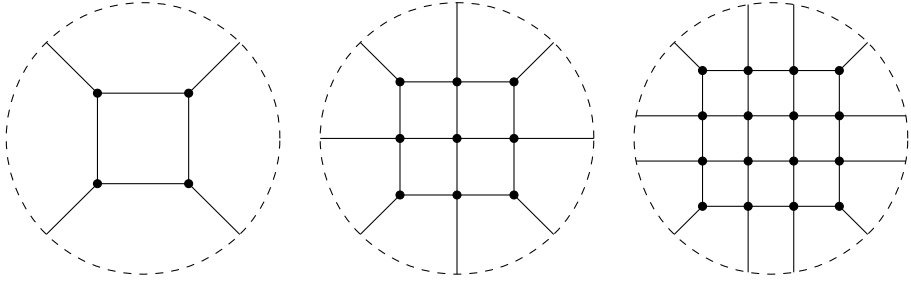

\centering
\begin{tikzgraph}[scale=1.2,thin]
\def\r{1.5}
\def\k{0.5}
\draw[dashed] (0,0) circle (\r);

\foreach\i in {1,...,4}
{
  \path (45+90*\i:\r) coordinate (c\i);
}
\draw (-\k,\k)--(-\k,-\k)
      (-\k,-\k)--(\k,-\k)
      (\k,-\k)--(\k,\k)
      (\k,\k)--(-\k,\k)
      (c1)--(-\k,\k)
      (c2)--(-\k,-\k)
      (c3)--(\k,-\k)
      (c4)--(\k,\k);

\foreach \i in {-1,1}
{
  \foreach \j in {-1,1}
  {
    \draw ({\i/2},{\j/2}) node[vertex] {};
  }
}
\end{tikzgraph}
\hfil
\begin{tikzgraph}[scale=0.75,thin]
\def\r{2.4}
\def\k{1}
\def\l{0}
\draw[dashed] (0,0) circle (\r);

\foreach\i in {-\k,...,\k}
{
  \path (\i,{sqrt(\r^2-abs(\i^2))}) coordinate (n\i)
        (\i,-{sqrt(\r^2-abs(\i^2))}) coordinate (s\i)
        (-{sqrt(\r^2-abs(\i^2))},\i) coordinate (e\i)
        ({sqrt(\r^2-abs(\i^2))},\i) coordinate (w\i);
}
\foreach\i in {1,...,4}
{
  \path (45+90*\i:\r) coordinate (c\i);
}
\foreach\i in {0}
{
  \draw (n\i)--(s\i)
        (e\i)--(w\i);
}
\draw (-\k,\k)--(-\k,-\k)
      (-\k,-\k)--(\k,-\k)
      (\k,-\k)--(\k,\k)
      (\k,\k)--(-\k,\k)
      (c1)--(-\k,\k)
      (c2)--(-\k,-\k)
      (c3)--(\k,-\k)
      (c4)--(\k,\k);

\foreach \i in {-\k,...,\k}
{
  \foreach \j in {-\k,...,\k}
  {
    \draw (\i,\j) node[vertex] {};
  }
}
\end{tikzgraph}
\hfil
\begin{tikzgraph}[scale=0.6,thin]
\def\r{3}
\def\k{1.5}
\draw[dashed] (0,0) circle (\r);

\foreach\i in {-3,-1,1,3}
{
  \path ({\i/2},{sqrt(\r^2-abs((\i/2)^2))}) coordinate (n\i)
        ({\i/2},-{sqrt(\r^2-abs((\i/2)^2))}) coordinate (s\i)
        (-{sqrt(\r^2-abs((\i/2)^2))},{\i/2}) coordinate (e\i)
        ({sqrt(\r^2-abs((\i/2)^2))},{\i/2}) coordinate (w\i);
}
\foreach\i in {1,...,4}
{
  \path (45+90*\i:\r) coordinate (c\i);
}
\foreach\i in {-1,1}
{
  \draw (n\i)--(s\i)
        (e\i)--(w\i);
}
\draw (-\k,\k)--(-\k,-\k)
      (-\k,-\k)--(\k,-\k)
      (\k,-\k)--(\k,\k)
      (\k,\k)--(-\k,\k)
      (c1)--(-\k,\k)
      (c2)--(-\k,-\k)
      (c3)--(\k,-\k)
      (c4)--(\k,\k);

\foreach \i in {-3,-1,1,3}
{
  \foreach \j in {-3,-1,1,3}
  {
    \draw ({\i/2},{\j/2}) node[vertex] {};
  }
}
\end{tikzgraph}
\caption{The graphs $G_2$, $G_3$ and $G_4$ embedded in the projective plane $\Pp$.}
\label{fig:grids}
\end{figure}

\begin{proof}[Proof of \Cref{thm:OCT-lower}]
The argument is quite similar to the argument given in~\cite{Ree99}
for \emph{Escher walls}, a related construction of Lov\'asz and
Schrijver. Assume for the sake of contradiction that $G_k$ has an odd
cycle transversal $T$ of size at most $k-1$. Then for some $\ell \in
[k]$, $T$ is disjoint from the $\ell$-th row $R_\ell$ of $G_k$ (all
the vertices $(i,\ell)$, with $i \in [k]$). Consider, for each $i \in
[k]$, the set of vertices $L_i=\{(i,j)\,|\,j\in [\ell]\}\cup
\{(k-i-1,k-j-1)\,|\,j\in [k-\ell-1]\}$. Note that the sets $L_i$, $i
\in [k]$, are vertex-disjoint, so $T$ is disjoint from one of them, say
$L_m$. It follows that the set of vertices $C=L_m \cup \{(i,\ell)\,|\,
m \le i \le k-m-1\}\subseteq L_m \cup R_\ell$, is disjoint from $T$.
Observe that $C$ induces an odd cycle, which contradicts the fact
that $T$ was an odd cycle transversal.
\end{proof}

\Cref{thm:OCT-upper} immediately implies the following almost optimal lower bound on
the independence number of projective quadrangulations, which may be
new.

\begin{corollary}
  Let $G$ be a non-bipartite projective quadrangulation on $n$
  vertices. Then $\alpha(G) \geq \left\lceil\tfrac12(n-\tfrac14-\sqrt{n-15/16}) \right\rceil$.
  Moreover, the graph $G_k$ satisfies
  $\alpha(G_k) = \frac n2-\frac{\sqrt{n}}2$.
\end{corollary}

\begin{proof}
  By \Cref{thm:OCT-upper} $G$ has an odd cycle transversal $S$ such
  that $|S|\le \tfrac14+\sqrt{n-15/16}$. The graph $G-S$ is bipartite
  on more than $n-\tfrac14-\sqrt{n-15/16}$ vertices, so at least one
  color class of $G-S$ has more than
  $\tfrac12\left(n-\tfrac14-\sqrt{n-15/16}\right)$ vertices.
  
  We now focus on $G_k$. Since it contains an odd cycle transversal
  $T$ of $k=\sqrt{n}$ vertices (for instance, take $T=\{(i,i)\,| i
  \in [k]\}$), it also contains a stable set of size
  $\tfrac12(n-\sqrt{n})$. We now prove that this bound is tight. For
  the sake of contradiction, assume that there is a stable set $S$
  of size more than $\frac n2-\frac{\sqrt{n}}2=\frac12
  (k^2-k)$. Assume first that $k$ is even, and for $0\le i \le k/2$,
  let $K_i=R_i \cup R_{k-1-i}$. Recall that the $\ell$-th row
  $R_\ell$ of $G_k$ consists of the vertices $(i,\ell)$, with $i \in
  [k]$. Observe that each $K_i$ contains a cycle of length $2k$, and
  since $S$ is a stable set, $|S \cap K_i| \le k$. If $|S\cap K_i| \le k-1$
  for every $i\in [k/2]$, then $S$ contains at most
  $\tfrac{k}2(k-1)=\frac12 (k^2-k)$ vertices, which is a contradiction.
  Therefore $|S \cap K_i| = k$ for some index $i\in [k/2]$. As $(0,i)$
  and $(k-1,k-1-i)$ are adjacent it follows that for each $j\in [k]$,
  $(j,i)\in S$ if and only if $(j,k-1-i) \in S$. Since $k$ is even,
  we also have that for each $j\in [k]$, $(j,i)\in S$ if and only if
  $(k-1-j,i)\not\in S$.

  Let $C_\ell$ be the $\ell$-th column of $G_k$, i.e., the vertices
  $(\ell,j)$ with $j \in [k]$. By the same argument as above, there
  exists an index $j$ such that $|S\cap L_j|=k$, where $L_j=C_j \cup C_{k-1-j}$.
  As before, we have $(j,i)\in S$ if and only if $(j,k-1-i) \not\in S$,
  which contradicts the previous paragraph.

  The proof of the case when $k$ is odd is quite similar and we therefore
  omit it. The only difference is that in this case the middle row
  $R_{\lfloor k/2\rfloor}$ and the middle column $C_{\lfloor k/2\rfloor}$ each
  induce a cycle on $k$ vertices, which therefore contains at most
  $\tfrac12(k-1)$ vertices of $S$.
\end{proof}

Again, it can be checked that for any integer $k\ge 2$,
$$\left\lceil\tfrac12(k^2-\tfrac14-\sqrt{k^2-15/16})\right\rceil=\frac{k^2}2-\frac{k}2,$$ so this shows that
our lower bound on $\alpha(G)$ is sharp for infinitely many values of $n$.

\section{Almost independent odd cycle transversals}\label{sec:jap}

Recall that by Theorem~\ref{thm:oddcycle}, every
non-bipartite projective quadrangulation contains an odd cycle of
length $O(\sqrt{n})$. This odd cycle is also an odd cycle
transversal, and by increasing its size by at most one, we can even make
sure that it induces a \emph{proper} subgraph of an odd cycle, i.e.,
a union of paths (in particular, a bipartite graph). Therefore,
every non-bipartite projective quadrangulation can be properly colored with colors
$1,2,3,4$ in such a way that only $O(\sqrt{n})$ vertices are colored 1
or 2.

Nakamoto and Ozeki~\cite{NO17} have asked
whether the $4$-coloring might be chosen in such a way that one color class has
size $1$ and another has size $o(n)$. We now give an
affirmative answer to their question for graphs with sublinear maximum
degree.

We start with a lemma, whose proof is quite similar to that of
Theorem~\ref{thm:oddcycle}. In what follows, the neighborhood of a vertex $v$ is denoted by $N(v)$.

\begin{lemma}\label{lem:neighborhood}
Let $G=(V,E)$ be a non-bipartite quadrangulation of $\Pp$ and $S$
an odd support set of $G$. Then there is a subset
$T \subseteq \bigcup_{v \in S} N(v)$
such that $G[T]$ contains a single edge, and $G-T$ is
bipartite.
\end{lemma}

\begin{proof}
Recall that the order of a support set $S$ of $G$ is the number of
pairs of consecutive vertices of $S$ that are adjacent in $G$.
We prove that there is a support set of order $1$ that is included in
$\bigcup_{i=1}^s N(v_i)$.

This support set can be obtained as follows. Choose a pair $v_i,v_{i+1}$ of
consecutive vertices of $S$  that are adjacent, and let
$v_j,v_{j+1}$ be the next pair of consecutive vertices that are
adjacent (possibly, $j=i+1$). Then all pairs of consecutive vertices
$v_k,v_{k+1}$, with $i<k<j$, are opposite. For each vertex $v_k$, $i<k\le j$,
let $s_k$ be a sequence of consecutive neighbors of $v_k$, in their circular order
around $v_k$, such that $s_{i+1}$ starts with $v_i$, $s_j$ ends with
$v_{j+1}$, and for any $i<k< j$, the last vertex of $s_k$ and the
first vertex of $s_{k+1}$ coincide (see Figure~\ref{fig:neighborhood}). We then replace the
sequence $v_i,v_{i+1},\ldots,v_{j+1}$ in $S$ by the concatenation of
the sequences $s_k$, for $i<k\le j$ (where the last vertex of each
sequence $s_k$ is identified with the first vertex of $s_{k+1}$). Note that
any two consecutive vertices in this new subsequence are opposite.

\smallskip

\begin{figure}[htbp]
\centering \includegraphics[scale=1]{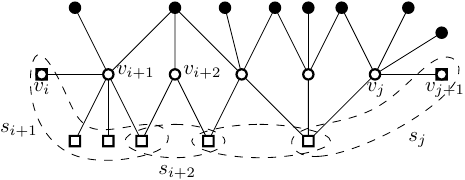}
\caption{Obtaining an odd support set of order 1. Vertices of $S$ are
  depicted with white dots and vertices of $T$ are depicted with white
  squares.} \label{fig:neighborhood}
\end{figure}

The operation above can be performed as long as the support set has
order at least two, and each such operation decreases the order of the
support set by precisely two. Since we started with an odd support set, we can repeat the operation, starting at the next pair (after $v_{j+1}$)
of consecutive vertices that are adjacent, until precisely one such pair
remains. The support set thus obtained has order $1$ (and is therefore odd)
and is a subset of $\bigcup_{v\in S} N(v)$, as desired. By
Lemma~\ref{lem:shortest}, there is a support set $T\subseteq
\bigcup_{v\in S} N(v)$ of order $1$ that is not self-intersecting and such
that there is a unique pair of adjacent vertices in $T$, and these
two vertices are consecutive. It follows that $T$ induces a subgraph
with a unique edge, and by Lemma~\ref{lem:parity}, $G-T$ is
bipartite.
\end{proof}

We now turn to the proof of \Cref{thm:sqrtD}. 

\begin{proof}[Proof of Theorem~\ref{thm:sqrtD}]
  Let $C$ be a shortest odd cycle in $G$, and let $\ell$ be the length of
  $C$. Assume first that $\ell< \sqrt{2n/\Delta}$. Since an odd cycle
  is an odd support set, it follows from Lemma~\ref{lem:neighborhood} that
  the union of the neighborhoods of the vertices of $C$ contain a set
  $T$ of vertices inducing a single edge, and such that $G-T$ is
  bipartite. Since $G$ has maximum degree at most $\Delta$, $T$
  contains at most $\ell \Delta< \sqrt{2n \Delta}$ vertices, as
  desired.

  Assume now that $\ell\ge \sqrt{2n/\Delta}$. Since $G$ has
  $2n-2$ edges, the dual graph $G^*$ of $G$ has a non-contractible cycle $C^*$ of length
  less than $2n/\ell$ by Lemma~\ref{lem:lins}. Let
  $(f_1,f_2,\ldots,f_k)$ be the faces of $G$ corresponding to the
  vertices of $C^*$. Note that any two consecutive faces $f_i,f_{i+1}$ in
  $C^*$ share an edge, call it $e_i$. For any edge $e_i$, we choose
  one of the endpoints $v_i$ of $e_i$ as follows: we start by choosing
  $v_1$ in $e_1$ arbitrarily, and for any $i>1$ we distinguish two
  cases. If $v_{i-1} \in e_i$, then we set $v_i=v_{i-1}$, and otherwise
  we choose for $v_i$ the vertex opposite to $v_{i-1}$ in $f_i$ (note
  that in this case such vertex is necessarily an endpoint of
  $e_i$). Note that $S=(v_i\,|\,1\le i \le
  k)$ is a support set in $G$, and $\rho(S)$ is homologous to
  $C^*$, and thus non-contractible.
  By Lemma~\ref{lem:parity}, $S$ is an odd support set (in
  particular $v_k$ and $v_1$ are adjacent, since all the other pairs of
  consecutive vertices of $S$ are opposite). By
  Lemma~\ref{lem:shortest}, $G$ contains a set $S'$ of less than
  $2n/\ell \le \sqrt{2n \Delta} $ vertices such that $G-S'$ is bipartite and $S'$ induces a
  subgraph of $G$ with a single edge. This concludes the proof of
  Theorem~\ref{thm:sqrtD}.
\end{proof}

Note that a subgraph with a single edge has a proper $2$-coloring such
that one of the color classes is a singleton. It follows that
$n$-vertex projective quadrangulations with
$\Delta=o(n)$ can be $4$-colored in such a way that one color class is a
singleton and another has size $o(n)$.

\section{Conclusion}\label{sec:conclusion}

We have seen that Theorem~\ref{thm:oddcycle2} is sharp for infinitely
many values of $n$. A natural question is whether the generalized
Mycielski graphs are the only extremal graphs.
As for the problem of finding a smallest odd cycle transversal, we
believe that Theorem~\ref{thm:OCT-upper} is not sharp and that the
right bound should be $\sqrt{n}$, which would be tight by
Theorem~\ref{thm:OCT-lower}.

It was proved by Tardif~\cite{Tar01} that the generalized Mycielki graphs have
\emph{fractional} chromatic number $2+o(\tfrac1{n})$, so a natural
question is whether the same holds for projective quadrangulations of
large edge-width. Note that it was proved by Goddyn~\cite{Goddyn}
(see also~\cite{DGMVZ05}), in the same spirit as the result of Youngs~\cite{You96}
on the chromatic number, that the \emph{circular} chromatic number of a projective
quadrangulation is either $2$ or $4$.

\bibliographystyle{plain}
\bibliography{OCT-quad}

\end{document}